\documentclass[12pt]{article}
\usepackage{amsmath,amsfonts,amssymb,amsthm,latexsym,pb-diagram}
\usepackage{enumerate}

\newtheorem{thm}{Theorem}
\newtheorem{defn}[thm]{Definition}
\newtheorem{lem}[thm]{Lemma}
\newtheorem{cor}[thm]{Corollary}

\begin{document}


\title{A group commutator involving the last distance matrix and dual distance matrix of a $Q$-polynomial distance-regular graph}
\author{Siwaporn  Mamart\\ \begin{small}{Department of Mathematics, Silpakorn University, Nakhon Pathom, Thailand}\end{small} \\
\begin{small}{mamart\_s@silpakorn.edu}\end{small}}
\date{}
\maketitle

\begin{abstract}
Let $\Gamma$ denote the  Hamming graph $H(D,r)$ with $r \geq 3$. 
Consider the  distance matrices $\{A_i\}_{i=0}^{D}$ of $\Gamma$. 
Fix a vertex $x$ of $\Gamma$, and consider the dual distance matrices  $\{A_i^{*}\}_{i=0}^{D}$  of $\Gamma$ with respect to $x$.
We investigate the  group commutator $A_{D}^{-1}A_{D}^{*-1}A_{D}A_{D}^{*}$.
We show that this matrix is diagonalizable.
We compute its eigenvalues and their eigenspaces.
Let $T$ denote the subconstituent algebra of $\Gamma$ with respect to $x$.
We describe the action of $A_{D}^{-1}A_{D}^{*-1}A_{D}A_{D}^{*}$ on each  irreducible $T$-module. 
\end{abstract}

\section{Introduction}

\hskip 19pt  Let $\Gamma$ denote a $Q$-polynomial distance-regular graph with diameter $D$ and distance matrices $\{A_i\}_{i=0}^{D}$. 
Fix a vertex $x$ of $\Gamma$, and let $\{A_i^{*}\}_{i=0}^{D}$ denote the dual distance matrices of $\Gamma$ with respect to $x$ (formal definitions will begin in Section 2).
It is known that  $A_{D}$ and  $A_{D}^{*}$ are invertible \cite[Theorem 6.6]{NT}. 
Motivated by this, we consider the group commutator $A_{D}^{-1}A_{D}^{*-1}A_{D}A_{D}^{*}$.
It is natural to ask, is this matrix diagonalizable, and if so, what are its eigenvalues and their eigenspaces.
In this paper we investigate this question.
For the sake of concreteness, we assume that $\Gamma$ is the  Hamming graph $H(D,r)$ with $r \geq 3$. 
Our results are summarized as follows.
We show that the product $A_{D}^{-1}A_{D}^{*-1}A_{D}A_{D}^{*}$ is diagonalizable with eigenvalues 
\begin{displaymath} (1-r)^s  \qquad -D \leq s \leq D. \end{displaymath}
For each eigenvalue $(1-r)^s$, we describe the corresponding eigenspace in terms of the split decomposition (see \cite{split}).
We show that the dimension of this eigenspace is
\begin{displaymath}  \displaystyle   \sum_{\substack{0 \leq i,j \leq D\\ i+j \geq D \\ j-i =s}}\binom{D}{i}\binom{i}{D-j}(r-2)^{i+j-D}. \end{displaymath}
It is known that the matrices $\{A_i\}_{i=0}^{D}$ form a  basis for the Bose-Mesner algebra $M$ of $\Gamma$. 
Similarly, the matrices $\{A_i^{*}\}_{i=0}^{D}$ form a  basis for the dual Bose-Mesner algebra $M^*$ of $\Gamma$ with respect to $x$.
Recall that $M$ and $M^*$ generate the subconstituent algebra $T$  of $\Gamma$ with respect to $x$.
Let $W$ denote an irreducible $T$-module.
We consider the action of $A_{D}^{-1}A_{D}^{*-1}A_{D}A_{D}^{*}$ on $W$.
We show that this action is diagonalizable with eigenvalues
\begin{center}
 $(1-r)^s$ \qquad $-d \leq s \leq d$, \quad $d-s$ is even,
\end{center}
where $d$ is the diameter of $W$.
For this action, we show that each eigenspace has dimension $1$.

This paper is organized as follows.
In Section 2, we discuss some basic facts about  $Q$-polynomial distance-regular graphs.
In Section 3, we give a detailed description of the subconstituent algebra $T$  for the complete graph $K_r$.
In Section 4, we focus on the  Hamming graph $H(D,r)$; this section contains our main results.

\section{Preliminaries}

\hskip 19pt Let $\mathbb{C}$ denote the complex number field.
Let $X$ denote a nonempty finite set. 
Let $Mat_{X}(\mathbb{C})$ denote the $\mathbb{C}$-algebra consisting of all matrices whose rows and columns are indexed by $X$ and whose entries are in $\mathbb{C}$.
Let $V = \mathbb{C}^{X}$ denote the vector space over $\mathbb{C}$ consisting of column vectors whose coordinates are indexed by $X$ and whose entries are in $\mathbb{C}$.
We observe that $Mat_{X}(\mathbb{C})$ acts on $V$ by left multiplication. 
We call $V$ the {\it standard module}.
We endow $V$ with the Hermitean inner product $\langle \; ,\,\rangle$ that satisfies  $\langle u,v\rangle=u^t\overline{v}$ for all $u,v\in{V}$. Here $t$ denotes the transpose, and  $\rule[0.3cm]{0.3cm}{0.1pt}$  denotes complex conjugation.
For each $y \in X$, let $\hat{y}$ denote the element of $V$ with a one in the $y$ coordinate and zero in all other coordinates.
We  observe that $\{ \hat{y} | y \in X \}$ is an orthonormal basis for $V$.
Let  $\mathbf{1}$ denote the all $1$'s vector in $V$.
Observe that $\mathbf{1} = \sum_{y \in X} \hat{y}$.

Throughout the paper, $\Gamma$ denotes a finite, undirected, connected graph with vertex set $X$ and path length distance function $\partial$.
Define $D = \max\{ \partial(x,y) | x,y \in X \}$. We call $D$ the {\it  diameter} of $\Gamma$.
%
Let $k$ denote a non-negative integer. 
Then $\Gamma$ is said to be {\it regular} with valency $k$ whenever every vertex of $\Gamma$ is adjacent to exactly $k$
distinct vertices of $\Gamma$. 
%
%
We say that $\Gamma$ is {\it distance-regular} whenever for all integers $h, i, j$ $(0 \leq h, i, j \leq D)$
and for all $x, y \in X$ with $\partial(x, y) = h$, the number
$p^h_{ij}:=|\{z \in X | \partial(x, z) = i  \ \text{and}\  \partial(y, z) = j \}|$
is independent of $x$ and  $y$.
The constants $p^h_{ij}$ $(0 \leq h, i, j \leq D)$ are called the {\it intersection numbers} of $\Gamma$.
For the rest of this paper, assume that $\Gamma$ is  distance-regular with $D \geq 1$.
Observe that $\Gamma$ is regular with valency $k=p^0_{11}$.
For each integer $i$ $(0 \leq i \leq D)$, let $A_i$ denote the matrix in $Mat_{X}(\mathbb{C})$ with $(x,y)$-entry
\begin{displaymath}
\displaystyle  (A_i)_{xy}
= \left\{ \begin{array}{ll}
1 & \textrm{if $\partial(x,y)=i$} \\[2pt]
0 & \textrm{if $\partial(x,y)\neq i$}
\end{array} \right.
\quad (x,y \in X).
\end{displaymath}
We call $A_i$ the  {$i$th} {\it  distance matrix} of $\Gamma$. 
For convenience, define $A_i = 0$ if $i < 0$ or $i > D$.
We abbreviate $A = A_1$, and call this the {\it adjacency matrix} of $\Gamma$.
Observe 
$$\begin{array}{rcl}
 A_0 &=&  I,\\ [2pt]
\displaystyle  \sum_{i=0}^{D} A_i &=&  J \qquad \qquad \; \text{($J$ = all $1$'s matrix)},\\ [10pt]
\overline{A_i} &=& A_i  \qquad \qquad(0 \leq i \leq D),\\
A_{i}^t &=& A_i \qquad \qquad (0 \leq i \leq D),\\
A_i A_j &=& \displaystyle  \sum_{h=0}^{D} p^h_{ij} A_h \; \quad (0 \leq i,j  \leq D).
\end{array}$$
Let $M$ denote the subalgebra of $Mat_{X}(\mathbb{C})$ generated by $A$. 
We call $M$ the {\it Bose-Mesner algebra} of $\Gamma$.
The matrices $A_0, A_1, \ldots  , A_D$ form a basis for $M$.
By \cite[p. 45]{BCN}, $M$ has a second basis $E_0, E_1, \ldots , E_D$ such that
$$\begin{array}{rcl}
 E_0 &=& |X|^{-1}J, \\ [2pt]
\displaystyle  \sum_{i=0}^{D} E_i &=&  I, \\ [10pt]
\overline{E_i} &=& E_i  \qquad \qquad(0 \leq i \leq D),\\
E_{i}^t &=& E_i \qquad \qquad (0 \leq i \leq D),\\
E_i E_j &=&  \delta_{ij} E_i \qquad \quad (0 \leq i,j  \leq D).
\end{array}$$
We call $E_0, E_1, \ldots , E_D$ the {\it primitive idempotents} of $\Gamma$.\\
Let $\circ$ denote the entry-wise product in $Mat_{X}(\mathbb{C})$. 
Observe 
\begin{center} $A_i \circ A_j = \delta_{ij} A_i \qquad \quad (0 \leq i,j  \leq D)$. \end{center}
Consequently, $M$ is closed under $\circ$. 
Therefore, there exist complex scalars $q^h_{i j}$ such that
\begin{center} $E_i \circ E_j = \displaystyle  |X|^{-1} \sum_{h=0}^{D} q^h_{ij} E_h \qquad \quad (0 \leq i,j  \leq D)$. \end{center}
By \cite[p. 170]{BIGGS}, the scalar $q^h_{i j}$ is real and nonnegative for $0 \leq h, i,j  \leq D$.
The $q^h_{i j}$ are called the {\it Krein parameters} of $\Gamma$.
The graph $\Gamma$ is said to be {\it $Q$-polynomial} with respect to the ordering $E_0, E_1, \ldots , E_D$ whenever the following hold for $0 \leq h, i, j \leq D$:
\begin{enumerate}
\item[\rm(i)] $q^h_{i j}=0$  if one of $h, i, j$ is greater than the sum of the other two;
\item[\rm(ii)] $q^h_{i j} \neq 0$  if one of $h, i, j$ equals the sum of the other two.
\end{enumerate}

For the rest of this paper, assume that $\Gamma$ is $Q$-polynomial with respect to  $E_0, E_1, \ldots , E_D$.
For the rest of this section, fix $x \in X$. 
For each integer $ i$ $(0 \leq i \leq D)$, let $E_i^{*}=E_i^{*}(x)$ denote the diagonal matrix in $Mat_{X}(\mathbb{C})$
with $(y,y)$-entry
\begin{displaymath}
\displaystyle  (E_i^{*})_{yy}
= \left\{ \begin{array}{ll}
1 & \textrm{if $\partial(x,y)=i$} \\[2pt]
0 & \textrm{if $\partial(x,y)\neq i$}
\end{array} \right.
\quad (y \in X).
\end{displaymath}
We call $E_i^{*}$ the {$i$th} {\it dual idempotent} of $\Gamma$ with respect to $x$.
For convenience, define $E_i^{*} = 0$ if $i < 0$ or $i > D$. Observe
$$\begin{array}{rcl}
\displaystyle  \sum_{i=0}^{D} E_i^{*}  &=&  I, \\ [15pt]
\overline{E_i^{*} } &=& E_i^{*}   \qquad \qquad(0 \leq i \leq D),\\
E_i^{*t}&=& E_i^{*}  \qquad \qquad (0 \leq i \leq D),\\
E_i^{*} E_j^{*} &=&  \delta_{ij} E_i^{*}  \qquad \quad (0 \leq i,j  \leq D).
\end{array}$$
The matrices $E_0^{*}, E_1^{*}, \ldots , E_D^{*}$ form a basis for a commutative subalgebra $M^*= M^*(x)$ of $Mat_{X}(\mathbb{C})$, called the {\it dual Bose-Mesner algebra of $\Gamma$ with respect to $x$} \cite[p. 378]{1paper}. 
We now give another basis for $M^*$. 
For each integer $ i$ $ (0 \leq i \leq D)$, let $A_i^{*}= A_i^{*}(x)$ denote the diagonal matrix in  $Mat_{X}(\mathbb{C})$  with $(y,y)$-entry
\begin{center} $(A_i^{*})_{yy} = |X| (E_i)_{xy} \qquad (y \in X)$. \end{center}
We call $A_i^{*}$ the {$i$th} {\it  dual distance matrix} of $\Gamma$ with respect to $x$ \cite[p. 379]{1paper}.
We abbreviate $A^* = A_1^*$, and call this the {\it dual adjacency matrix} of $\Gamma$ with respect to x.
The matrices $A_0^{*}, A_1^{*}, \ldots , A_D^{*}$ form a  basis for  $M^*$ such that
$$\begin{array}{rcl}
 A_0^{*} &=&  I, \\ [2pt]
\displaystyle  \sum_{i=0}^{D} A_i^{*} &=&  |X| E_0^{*},\\ [15pt]
\overline{A_i^{*}} &=& A_i^{*}  \qquad \qquad(0 \leq i \leq D),\\
A_{i}^{*t} &=& A_i^{*} \qquad \qquad (0 \leq i \leq D),\\ 
A_i^{*} A_j^{*} &=& \displaystyle  \sum_{h=0}^{D} q^h_{ij} A_h^{*} \; \quad (0 \leq i,j  \leq D).
\end{array}$$
Observe that the algebra $M^*$ is  generated by  $A^*$  \cite[Lemma 3.11]{1paper}.

\begin{defn} {\rm \cite[Definition 3.3]{1paper} }
{\rm Let $T=T(x)$ denote the subalgebra of $Mat_{X}(\mathbb{C})$  generated by $M$ and $M^*$.
We call $T$ the {\it subconstituent algebra (or Terwilliger algebra)} of $\Gamma$  with respect to $x$.}
\end{defn}

We observe that $T$ is generated by $A$ and $A^*$. Moreover $T$ has finite dimension.
By \cite[Lemma 3.4]{1paper}, the algebra $T$ is semisimple.

\begin{defn}
{\rm An element $C \in T$ is called {\it central} whenever $CB = BC$ for all $B \in T$. 
Define $Z(T) = \{C \in T | \text{$C$ is central}\}$. We call $Z(T)$ the {\it center} of $T$.}
\end{defn}

\begin{defn}
{\rm By a {\it $T$-module} we mean a subspace $W \subseteq V$ such that $BW \subseteq W$ for all $B \in T$.}
\end{defn}

Let $W$ denote a $T$-module. 
Then $W$ is said to be {\it irreducible} whenever $W$ is nonzero, and $W$ contains no $T$-modules other than $0$ and $W$.
By \cite[Corollary 6.2]{JGo} any $T$-module is an orthogonal direct sum of irreducible $T$-modules. In particular the standard module $V$ is an orthogonal direct sum of irreducible $T$-module. 
By \cite[Lemma 3.3]{Curtin}, any two non-isomorphic irreducible $T$-modules are orthogonal.

Let $W$ denote an irreducible $T$-module.
By the {\it endpoint} of $W$ we mean $\min\{i|0 \leq i \leq D, E_i^{*}W \neq 0\}$.
By the {\it diameter} of $W$ we mean $|\{i|0 \leq i \leq D, E_i^{*}W \neq 0\}|-1$.
By the {\it dual endpoint} of $W$ we mean $\min\{i|0 \leq i \leq D, E_iW \neq 0\}$.
By the {\it dual diameter} of $W$ we mean $|\{i|0 \leq i \leq D, E_iW \neq 0\}|-1$.
The diameter of $W$ is equal to the dual diameter of $W$ \cite[Corollary 3.3]{APas1}.
The $T$-module $W$ is {\it thin} whenever $\dim(E_i^{*}W) \leq 1$ for $0 \leq i \leq D$.
In this case $\dim(E_i W) \leq 1$ for $0 \leq i \leq D$ \cite[Lemma 3.9]{1paper}.
By the {\it displacement} of $W$ we mean the integer $\rho + \tau + d - D$, where $\rho, \tau, d$ denote the endpoint, dual endpoint, and diameter of $W$, respectively.
By \cite[Theorem 8.3, Theorem 8.4]{Egge} there is a unique irreducible $T$-module $W$ with $E_0^{*}W \neq 0$ and $E_0 W \neq 0$, called the {\it primary} $T$-module.
The primary $T$-module has basis $A_0 \hat{x}, \ldots, A_D \hat{x}$ \cite[Lemma 3.6]{1paper}.
The primary $T$-module is thin \cite[Theorem 8.4]{Egge}.

By \cite[Theorem 6.6]{NT} $A_{D}$ and  $A_{D}^{*}$ are invertible. 
Motivated by this, we consider the group commutator $A_{D}^{-1}A_{D}^{*-1}A_{D}A_{D}^{*}$.
For the sake of concreteness, we focus on the  Hamming graph $H(D,r)$ with $r \geq 3$. 
Before we discuss $H(D,r)$, it is useful to discuss the complete graph $K_r$.
We do this in the next section.

\section{The complete graph $K_r$}
\hskip 19 pt From now on, fix an integer $r \geq 3$ and assume that $|X|=r$.

\begin{defn} \label{DefKr}
{\rm The {\it complete graph} $K_r$ has vertex set $X$, and any two distinct vertices are adjacent.}
\end{defn}

For the rest of this section, assume that $\Gamma$ is the complete graph $K_r$.
We have $E_0 = r^{-1}J$, $E_1 = I - E_0 = I - r^{-1}J$  and $A= J-I = (r-1)E_0 - E_1$.
For the rest of this section, fix  $x \in X$ and let $T=T(x)$ be the corresponding subconstituent algebra.
We have $A^* = (r-1)E_0^{*}- E_1^{*}$.

We have a comment about writing matrices in  $Mat_{X}(\mathbb{C})$. When we list the elements of $X$  we list $x$ first.
We have 

\begin{equation}
E_0^{*} = 
\begin{bmatrix}
   1 & 0 & 0 & \cdots & 0 \\
    0 & 0 & 0 & \cdots & 0 \\
 0 & 0 & 0 & \cdots & 0 \\
 \vdots & \vdots & \vdots & \ddots & \vdots \\
  0 & 0 & 0 & \cdots & 0 
\end{bmatrix}, \qquad
E_1^{*} = 
\begin{bmatrix}
   0 & 0 & 0 & \cdots & 0 \\
   0  & 1 & 0 & \cdots & 0 \\
 0 & 0 & 1 & \cdots & 0 \\
 \vdots & \vdots & \vdots & \ddots & \vdots \\
  0 & 0 & 0 & \cdots & 1 
\end{bmatrix}.
\label{eq:E0*}
\end{equation}

\begin{lem} \label{genT}
For the graph $K_r$, the matrices $E_0$ and $E_0^{*}$ generate $T$.
\end{lem}

\begin{proof}
The algebra $T$ is generated by $A$ and $A^*$. Also $A= r E_0 - I$ and $A^* = r E_0^* - I$.
\end{proof}

\begin{lem} \label{3terms}
For the graph $K_r$, the following hold:
\begin{enumerate}
\item[\rm(i)] $r E_0 E^{*}_0 E_0 =  E_0$.
\item[\rm(ii)] $rE^{*}_0 E_0 E^{*}_0 = E^{*}_0$.
\end{enumerate}
\end{lem}

\begin{proof}
By matrix multiplication using (\ref{eq:E0*}) and $E_0 = r^{-1}J$.
\end{proof}

Note that for the complete graph $K_r$, the primary $T$-module has dimension 2. 
Up to isomorphism, there exists a unique non-primary irreducible $T$-module, and this has dimension 1.
In particular, every irreducible $T$-module is thin.

\begin{lem} \label{Tiso}
For the graph $K_r$, the algebra $T$ is isomorphic to $Mat_{2}(\mathbb{C}) \oplus \mathbb{C}$.
\end{lem}

\begin{proof}
By the comment below Lemma \ref{3terms}, and since $T$ is semisimple. 
\end{proof}

\begin{lem} \label{baseT}
For the graph $K_r$,  the following hold:
\begin{enumerate}
\item[\rm(i)] $\dim(T) =5$. 
\item[\rm(ii)] $I, E_0, E^{*}_0, E_0 E^{*}_0, E^{*}_0 E_0$ form a basis for $T$. 
\end{enumerate}
\end{lem}

\begin{proof}
${\rm(i)}$  By Lemma \ref{Tiso}. \\ 
${\rm(ii)}$  By Lemma \ref{genT},  Lemma \ref{3terms} and (i) above.
\end{proof}

\begin{defn}
{\rm Define $e_0 \in Mat_X(\mathbb{C})$ to be the projection onto the primary $T$-module. Define $e_1 \in Mat_X(\mathbb{C})$ to be the projection onto the span of the non-primary irreducible $T$-modules.
}
\end{defn}

Observe that 
\begin{align} \label{e0e1} V=e_0V + e_1 V  \quad \text{(orthogonal direct sum)}. \end{align} 

\begin{lem}  \label{basee0V}
For the graph $K_r$, the following hold:
\begin{enumerate}
\item[\rm(i)] $E_0 V =  \mathbb{C} \mathbf{1}$ and $E^{*}_0 V =  \mathbb{C} \hat{x}$.
\item[\rm(ii)] the vectors $\hat{x}, \mathbf{1}$ form a basis for $e_0V$.
\item[\rm(iii)] $e_0 V = E_0 V + E^{*}_0 V$ (direct sum).
\end{enumerate}
\end{lem}

\begin{proof}
${\rm(i)}$   By matrix multiplication.\\
${\rm(ii)}$  Observe that $\hat{x}, \mathbf{1}$  are linearly independent and contained in the primary $T$-module.
Since the primary $T$-module has dimension 2, the result follows.\\
${\rm(iii)}$ By (i) and (ii).
\end{proof}

\begin{lem} \label{matrepAA*}
For the graph $K_r$, with respect to the basis $\hat{x}, \mathbf{1}$  the matrices representing $A, A^*$ are
 \[  A : 
\begin{bmatrix}
   -1 & 0  \\
    1 & r-1
\end{bmatrix},  \qquad 
A^{*} : 
\begin{bmatrix}
   r-1 & r  \\
   0  & -1
\end{bmatrix}.\]
\end{lem}

\begin{proof}
By matrix multiplication using  $A= J-I$ and $A^* = (r-1)E_0^{*}- E_1^{*}$.
\end{proof}

\begin{lem} \label{e1V}
For the graph $K_r$,  the following hold:
\begin{enumerate}
\item[\rm(i)] $\dim(e_1 V) =r-2$.
\item[\rm(ii)] $e_1V = span\{\hat{y}-\hat{z} \;|\; y,z \in $X$ \;\text{and}\; x,y,z \;\text{are mutually distinct}\}$.
\end{enumerate}
\end{lem}

\begin{proof}
${\rm(i)}$ View $e_1V = (e_0V)^\perp$. By Lemma \ref{basee0V}, $\dim(e_1 V) =r-2$.\\
${\rm(ii)}$ Define $S = span\{\hat{y}-\hat{z} | y,z \in X \;\text{and}\; x,y,z \;\text{are mutually distinct}\}$.
We will show that $S = e_1V$.
First we show that  $S \subseteq e_1V$.
Let $y, z \in X$  such that $x, y, z$ are mutually distinct. 
Then $\langle \hat{y} - \hat{z}, \hat{x}\rangle = 0$ and $\langle \hat{y} - \hat{z}, \mathbf{1}\rangle = 0$. 
The vectors $\hat{x}$ and $\mathbf{1}$ span $e_0V$ by Lemma \ref{basee0V}, so $\langle \hat{y} - \hat{z}, e_0V \rangle =0$.
Now $ \hat{y}-\hat{z} \in e_1 V$ by (\ref{e0e1}).
So $S \subseteq e_1V$.
To show equality, we show that $S$ has dimension $r-2$. 
To do this, we display $r-2$ linearly independent vectors in $S$.
Fix $y \in X$ with $y \neq x$. There exist $r-2$ vertices $z \in X$ such that $x, y, z$ are mutually distinct.
Observe that $\{\hat{y}-\hat{z} | z \in X \;\text{and}\; x,y,z \;\text{are mutually distinct}\}$ are linearly independent vectors in $S$.
Therefore $\dim(S) =r-2$.
By these comments, $S = e_1V$.
\end{proof}

%
%
%
%

\begin{lem}
For the graph $K_r$, the following hold:
\begin{enumerate}
\item[\rm(i)] $e_0^2 = e_0$.
\item[\rm(ii)] $e_1^2 =e_1$.
\item[\rm(iii)] $e_0 e_1 = e_1 e_0 = 0$.
\item[\rm(vi)]$e_0 + e_1 = I$.
\end{enumerate}
\end{lem}

\begin{proof}
By construction of $e_0$ and $e_1$.
\end{proof}

\begin{lem} \label{prop e0}
For the graph $K_r$, the following hold:
\begin{enumerate}
\item[\rm(i)] \[ e_0 = 
\begin{bmatrix}
   1      & 0 & 0 & \cdots & 0 \\
    0       & \frac{1}{r-1} & \frac{1}{r-1} & \cdots & \frac{1}{r-1} \\
 \vdots & \vdots & \vdots & \ddots & \vdots \\
  0       & \frac{1}{r-1} & \frac{1}{r-1} & \cdots & \frac{1}{r-1} 
\end{bmatrix}.
\]
\item[\rm(ii)] $\overline{e_0} = e_0$.
\item[\rm(iii)] $e_0^t = e_0$.
\item[\rm(iv)] $e_0 = \frac{n}{n-1}(E_0+E^*_0 - E_0 E^{*}_0 - E^{*}_0 E_0)$.
\item[\rm(v)] $e_0, e_1 \in T$.
\end{enumerate}
\end{lem}

\begin{proof}
{\rm(i)}--{\rm(iv)} By construction.\\
${\rm(v)}$ By ${\rm(iv)}$ and Lemma \ref{baseT}, $e_0 \in T$. Since $e_0 + e_1 = I$, $e_1 \in T$.
\end{proof}

\begin{lem} \label{commute}
For the graph $K_r$, the following hold:
\begin{enumerate}
\item[\rm(i)] $e_0 E^{*}_0 =  E^{*}_0 e_0 =  E^{*}_0$.
\item[\rm(ii)] $e_0 E_0 =  E_0 e_0 =  E_0$.
\end{enumerate}
\end{lem}

\begin{proof}
${\rm(i)}$  By Lemma \ref{prop e0}(i) and (\ref{eq:E0*}).\\
${\rm(ii)}$  Similar to (i).
\end{proof}

\begin{cor}
For the graph $K_r$, the elements $e_0, e_1$  form a basis for $Z(T)$.
\end{cor}

\begin{proof}
Combining Lemma \ref{genT} and Lemma \ref{commute}, we have $e_0 \in Z(T)$.
Since $e_0 + e_1 = I$, $e_1 \in Z(T)$.
By construction, $e_0$, $e_1$ are linearly independent.
By Lemma \ref{Tiso}, the dimension of $Z(T)$ is 2.
The result follows.
%
\end{proof}

\begin{lem} \label{lemofL15}
For the graph $K_r$, the subspace $e_1V$ is orthogonal to each of $E_0 V$ and $E^{*}_0 V$.
\end{lem}

\begin{proof}
By (\ref{e0e1}) and Lemma \ref{basee0V}(iii).
\end{proof}

\begin{lem}  \label{dsumofV}
For the graph $K_r$,
\begin{align} \label{eqn:2}
V = E^{*}_{0}V + E_{0}V + e_{1}V \qquad \text{(direct sum)}.
\end{align}
\end{lem}

\begin{proof}
By (\ref{e0e1}) and Lemma \ref{basee0V}(iii).
\end{proof}

\begin{lem} \label{e1V=}
For the graph $K_r$,   
\begin{center} 
$e_1V = E_1 V \cap E_1^{*} V$.
\end{center}
\end{lem}

\begin{proof}
We have the orthogonal direct sums  $V = E_0 V + E_1 V$ and $V = E^{*}_0 V + E^{*}_1 V$.
Using these and Lemmas \ref{lemofL15}, \ref{dsumofV},
\begin{center} 
 $e_1V = {(E_0 V + E_0^{*} V)}^{\perp} = (E_0 V)^{\perp} \cap (E_0^{*} V)^{\perp} = E_1 V \cap E_1^{*} V$.
\end{center}
\end{proof}

\begin{lem} \label{AA*acte1V}
For the graph $K_r$,  each of $A, A^*$ acts on $e_1 V$ as $-I$.
\end{lem}

\begin{proof}
Since $A$ acts on $E_1 V$ as $-I$ and  $A^{*}$ acts on $E_1^{*} V$ as $-I$, each of $A, A^*$ acts on $E_1 V \cap E_1^{*} V$ as $-I$.
By Lemma \ref{e1V=}, the result follows.
\end{proof}

\begin{lem} \label{prodKr}
For the graph $K_r$, the matrix ${A}^{-1}{A}^{*-1}{A}{A}^{*}$ is diagonalizable.
Its eigenspaces are $E^{*}_{0}V, E_{0}V,  e_{1}V$. The corresponding eigenvalues are
\begin{center} $1-r$, \quad $(1-r)^{-1}$, \quad $1$.  \end{center}
\end{lem}

\begin{proof}
We first show that $\hat{x}$ (resp. $\mathbf{1}$)  is an eigenvector for ${A}^{-1}{A}^{*-1}{A}{A}^{*}$ with eigenvalue $1-r$ (resp. $(1-r)^{-1}$).
Let $B$ (resp. $B^{*}$) denote the matrix representing $A$ (resp. $A^*$) with respect to the basis $\hat{x}, \mathbf{1}$ as in Lemma \ref{matrepAA*}.
By matrix multiplication,
 \[  {B}^{-1}{B}^{*-1}{B}{B}^{*} =
\begin{bmatrix}
   1-r & 0  \\
    0 & \frac{1}{1-r}
\end{bmatrix}.\]
The results follow in view of Lemma \ref{basee0V}(i) and  Lemmas \ref{dsumofV}, \ref{AA*acte1V}.
\end{proof}


\section{The Hamming graph $H(D,r)$}

\hskip 19pt In this section we give our main results, which are about the Hamming graph $H(D,r)$.

\begin{defn}
{\rm For an integer $D \geq 1$, the {\it  Hamming graph} $H(D, r)$ has vertex set the Cartesian product of $D$ copies of $X$, with two vertices adjacent whenever  they differ in precisely one coordinate.}
\end{defn}

By \cite[p. 27]{BCN} the Hamming graph $H(D, r)$ is distance-regular with diameter $D$. By \cite[p. 261]{BCN} $H(D,r)$ is $Q$-polynomial. 
Observe that for $D =1$, $H(D,r)$ is the complete graph $K_r$.


In order to distinguish between $H(D,r)$ and $K_r$, we use the following notations. 
When we talk about $H(D,r)$, we write everything in bold. For example, for $H(D,r)$  the vertex set is denoted by $\mathbb{X}$ and the standard module is denoted by  $\mathbb{V}$.
On the other hand, when we discuss $K_r$, we retain the notation that we set up earlier. 
Observe that  $\mathbb{X} = X \times X \times \cdots \times X$ (Cartesian product of $D$ terms). 
Therefore we have a tensor product $\mathbb{V} = V^{\otimes D}$.
The algebras $\mathbb{M}^*$ and $\mathbb{T}$ are with respect to the vertex $(x,x, \ldots, x)$ where $x$ is from below Definition \ref{DefKr}. 

We  now state our first main result.

\begin{thm} \label{lem8}
For  the graph $H(D,r)$ the matrix $\mathbb{A}_{D}^{-1}\mathbb{A}_{D}^{*-1}\mathbb{A}_{D}\mathbb{A}_{D}^{*}$ is diagonalizable, with eigenvalues 
\begin{center} $(1-r)^s$  \qquad $-D \leq s \leq D$. \end{center}
\end{thm}

We will prove Theorem \ref{lem8} after Theorem \ref{lem13}.

We now consider for each eigenvalue of $H(D,r)$, what is the corresponding eigenspace and its dimension?
To do this, it is convenient to bring in the split decomposition of $\mathbb{V}$ (see \cite{split}).
We now recall this decomposition.
Using $\mathbb{V} = V^{\otimes D}$ and (\ref{eqn:2}),
\begin{align} \label{eqn:st}
 \mathbb{V} = (E^{*}_{0}V+ E_{0}V+e_{1}V)^{\otimes D}.
\end{align}
Expanding (\ref{eqn:st}) we obtain 
\begin{align} \label{eqn:1}
 \mathbb{V} =   \sum V_1 \otimes V_2 \otimes \cdots \otimes V_D \quad \text{(direct sum)},
\end{align}
where the sum is over all sequences $V_1,V_2, \ldots, V_D$ of elements taken from $E_0^{*}V, E_{0}V, e_{1}V. $
For each summand in (\ref{eqn:1}) define
\begin{align} \label{alpha} \alpha = |\{j|1 \leq j \leq D, V_j = E_0^{*}V\}|, \end{align}
\begin{align} \label{beta} \beta = |\{j|1 \leq j \leq D, V_j = E_0V\}|,\end{align}
\begin{align} \label{eta} \eta = |\{j|1 \leq j \leq D, V_j = e_1V\}|.\end{align}
We call $\eta$ the {\it displacement}.\\
By construction 
\begin{align} \label{eqn:3} \alpha + \beta + \eta = D. \end{align}

\begin{defn} \label{defVeta}
{\rm For $0 \leq  \eta  \leq  D$  let $\mathbb{V}_\eta$ denote the sum of the terms in (\ref{eqn:1}) that have displacement $\eta$.}
\end{defn}

\begin{lem}  
For $0 \leq  \eta  \leq  D$, 
\begin{align} \label{sumVeta}
\mathbb{V}= \displaystyle  \sum_{\eta=0}^{D}   {\mathbb{V}}_{\eta} \quad \text{(orthogonal direct sum)}.
\end{align}
\end{lem}

\begin{proof}
By (\ref{eqn:1}) along with Lemma \ref{lemofL15} and the construction.
\end{proof}

\begin{defn}  \label{defVij}
{\rm For $-1 \leq i,j \leq D$ define 
\begin{center} $\mathbb{V}_{ij} = (\mathbb{E}_0^{*}\mathbb{V} + \mathbb{E}_1^{*}\mathbb{V} + \cdots + \mathbb{E}_i^{*}\mathbb{V}) \cap (\mathbb{E}_0\mathbb{V} + \mathbb{E}_1\mathbb{V} + \cdots + \mathbb{E}_j\mathbb{V})$.\end{center}}
\end{defn}
\noindent Observe that $\mathbb{V}_{ij} =0$ if $i=-1$ or $j=-1$. Also for $0 \leq i,j \leq D$,
$\mathbb{V}_{i-1,j} \subseteq \mathbb{V}_{ij}$  and  $ \mathbb{V}_{i,j-1} \subseteq \mathbb{V}_{ij}$. 
So $\mathbb{V}_{i-1,j}+\mathbb{V}_{i,j-1} \subseteq \mathbb{V}_{ij}$.

\begin{defn} \label{defVijtilde}
{\rm For $0 \leq i,j \leq D$ define 
\begin{center}{\it $\widetilde{\mathbb{V}}_{ij}$} = orthogonal complement of $ \mathbb{V}_{i-1,j}+ \mathbb{V}_{i,j-1}$ in $ \mathbb{V}_{ij}$.\end{center}}
\end{defn}
\noindent By \cite[Corollary 5.8]{split},
\begin{align} \label{sumVij}
\mathbb{V}= \displaystyle  \sum_{i=0}^{D} \sum_{j=0}^{D}   \widetilde{\mathbb{V}}_{ij} \quad \text{(direct sum)}.
\end{align}
By Lemma \ref{lemofL15} and Definitions \ref{defVij}, \ref{defVijtilde}, we have
\begin{align} \label{Vijtilde}
\widetilde{\mathbb{V}}_{ij} = \mathbb{V}_{ij} \cap \mathbb{V}_{\eta} \qquad \text{where}\; \eta = i+j-D.
\end{align}

\begin{lem} \label{lem95}
For $0 \leq i,j \leq D$,
 $\widetilde{\mathbb{V}}_{ij}$ is the sum of the terms in (\ref{eqn:1}) such that 
\begin{align} \label{abeta}
\alpha = D-i, \qquad \quad \beta =D-j, \qquad \quad \eta  = i+j - D.
\end{align}
\end{lem}

\begin{proof}
Let $\mathbb{V}'_{ij}$ denote the sum of the terms in (\ref{eqn:1}) that satisfy  (\ref{abeta}).
We show that $\mathbb{V}'_{ij} = \widetilde{\mathbb{V}}_{ij}$.
We first show that $\mathbb{V}'_{ij} \subseteq \widetilde{\mathbb{V}}_{ij}$.
Let $u$ denote a summand in (\ref{eqn:1}) that satisfies  (\ref{abeta}).
We show that $u \subseteq \widetilde{\mathbb{V}}_{ij}$.
Without loss of generality, $u =   (E_0^{*}V) ^{\otimes \alpha} \otimes (e_1V) ^{\otimes \eta} \otimes (E_0V) ^{\otimes \beta} $.
Observe that 
$$\begin{array}{rcl}
u &\subseteq& (E_0^{*}V) ^{\otimes \alpha} \otimes V^{\otimes i} \qquad \qquad \qquad \text{since $\beta + \eta =i$} \\
&=& (E_0^{*}V) ^{\otimes \alpha} \otimes (E_0^{*}V+E_1^{*}V)^{\otimes i}\\
&\subseteq& \mathbb{E}_0^{*}\mathbb{V} + \mathbb{E}_1^{*}\mathbb{V} + \cdots + \mathbb{E}_i^{*}\mathbb{V}.
\end{array}$$
Similarly, 
$$\begin{array}{rcl}
u &\subseteq&  V^{\otimes j} \otimes (E_0V) ^{\otimes \beta}  \qquad \qquad \qquad  \text{since $\alpha + \eta =j$}\\
&=& (E_0V+E_1V)^{\otimes j}  \otimes  (E_0V) ^{\otimes \beta}\\
&\subseteq& \mathbb{E}_0\mathbb{V} + \mathbb{E}_1\mathbb{V} + \cdots + \mathbb{E}_j\mathbb{V}.
\end{array}$$
By the above comments and Definition \ref{defVij}, $u \subseteq {\mathbb{V}}_{ij}$.
By construction $u \subseteq \mathbb{V}_{\eta}$. 
Thus $u \subseteq \mathbb{V}_{ij} \cap \mathbb{V}_{\eta} = \widetilde{\mathbb{V}}_{ij}$.
We have shown 
\begin{align} \label{V'ij} \mathbb{V}'_{ij} \subseteq \widetilde{\mathbb{V}}_{ij}  \qquad \qquad 0 \leq i,j \leq D. \end{align}
By (\ref{eqn:1}) and the construction,
\begin{align} \label{sumV'ij}
\mathbb{V}= \displaystyle  \sum_{i=0}^{D} \sum_{j=0}^{D}   \mathbb{V}'_{ij} \quad \text{(direct sum)}.
\end{align}
Combining (\ref{sumVij}), (\ref{V'ij}) and (\ref{sumV'ij}), we obtain $\mathbb{V}'_{ij} = \widetilde{\mathbb{V}}_{ij}$ for $0 \leq i,j \leq D$.
The result follows.
\end{proof}

\begin{lem} \label{lemVeta}
For $0 \leq \eta \leq D$,
\begin{align} \label{sumVetatilde}
\mathbb{V}_{\eta} = \displaystyle  \sum_{\substack{0 \leq i,j \leq D\\ i+j = \eta + D}} \widetilde{\mathbb{V}}_{ij}.
\end{align}
\end{lem}

\begin{proof}
By Lemma \ref{lem95}, the sum (\ref{sumVetatilde}) is equal to the sum of the terms in (\ref{eqn:1}) that have displacement $\eta$.
By Definition \ref{defVeta}, that sum is equal to $\mathbb{V}_{\eta}$. 
\end{proof}

\begin{lem} \label{lem12}
\vskip 2pt
For the graph $H(D,r)$, the following hold for $0 \leq i,j \leq D$.
\begin{enumerate}
\item[\rm(i)] Assume that $i+j < D$. Then $\widetilde{\mathbb{V}}_{ij} = 0$.
\item[\rm(ii)] Assume that $i+ j \geq D$. Then $\dim( \widetilde{\mathbb{V}}_{ij}) = \binom{D}{i}\binom{i}{D-j}(r-2)^{i+j-D}$.
\end{enumerate}
%
\end{lem}

\begin{proof}
We refer to the description of $ \widetilde{\mathbb{V}}_{ij}$ from Lemma \ref{lem95}.\\
\rm(i) Suppose that $\widetilde{\mathbb{V}}_{ij} \neq 0$.
Then there exists at least one term in  (\ref{eqn:1}) that satisfies (\ref{abeta}).
For this term, the displacement $\eta$ is non-negative by (\ref{eta}), so by (\ref{abeta}) we have $i+ j \geq D$, a contradiction.
Therefore $\widetilde{\mathbb{V}}_{ij} = 0$.\\
\rm(ii) Consider a summand in (\ref{eqn:1}) that contributes to $\widetilde{\mathbb{V}}_{ij}$.
This summand is described below (\ref{eqn:1}) and in Lemma \ref{lem95}.
In this summand, the 1-dimensional subspace $E_0^{*}V$ appears with multiplicity $\alpha$, the 1-dimensional subspace $E_0V$ appears with multiplicity $\beta$, and the $(r-2)$-dimensional subspace $e_1V$ appears with multiplicity $\eta$.
Therefore this summand has dimension $(r-2)^\eta$.
Also, the number of summands  with any given $\alpha, \beta, \eta$ is $\binom{D}{\alpha}\binom{D-\alpha}{\beta}\binom{D-\alpha-\beta}{\eta}$.
Hence
$$\begin{array}{rcl}
\dim( \widetilde{\mathbb{V}}_{ij}) &=&
\binom{D}{\alpha}\binom{D-\alpha}{\beta}\binom{D-\alpha-\beta}{\eta}(r-2)^{\eta}\\[5pt]
&=& \binom{D}{i}\binom{i}{D-j}(r-2)^{i+j-D}.
\end{array}$$
\end{proof}
%

\begin{lem} \label{prod act}
For the graph $H(D,r)$ and $0 \leq i,j \leq D$, the matrix $\mathbb{A}_{D}^{-1}\mathbb{A}_{D}^{*-1}\mathbb{A}_{D}\mathbb{A}_{D}^{*}$ acts on $\widetilde{\mathbb{V}}_{ij}$ as  $(1-r)^{j-i}I$.
\end{lem}

\begin{proof}
We view $\mathbb{A}_{D} = A^{\otimes D}$ and $\mathbb{A}_{D}^{*} = (A^{*})^{\otimes D}$.
Thus $\mathbb{A}_{D}^{-1}\mathbb{A}_{D}^{*-1}\mathbb{A}_{D}\mathbb{A}_{D}^{*} = ({A}^{-1}{A}^{*-1}{A}{A}^{*})^{\otimes D}$.
We refer to the description to $\widetilde{\mathbb{V}}_{ij}$ from Lemma \ref{lem95}.
On each summand in  (\ref{eqn:1}) that satisfies (\ref{abeta}), $\mathbb{A}_{D}^{-1}\mathbb{A}_{D}^{*-1}\mathbb{A}_{D}\mathbb{A}_{D}^{*}$ acts as $(1-r)^{\alpha - \beta}$ in view of Lemma \ref{prodKr}. 
By (\ref{abeta}) we have $\alpha - \beta = j-i$.
The result follows.
\end{proof}

\begin{thm} \label{lem13}
For the graph $H(D,r)$ and for $-D \leq s \leq D$, the subspace
\begin{align} \label{sumVtilde} \displaystyle   \sum_{\substack{0 \leq i,j \leq D\\ i+j \geq D \\ j-i =s}}\widetilde{\mathbb{V}}_{ij}\end{align} 
is an eigenspace for the matrix $\mathbb{A}_{D}^{-1}\mathbb{A}_{D}^{*-1}\mathbb{A}_{D}\mathbb{A}_{D}^{*}$. 
The  dimension of this   eigenspace is
\begin{center}$ \displaystyle   \sum_{\substack{0 \leq i,j \leq D\\ i+j \geq D \\ j-i =s}}\binom{D}{i}\binom{i}{D-j}(r-2)^{i+j-D}$.\end{center}
The corresponding  eigenvalue is $(1-r)^s$.
\end{thm}

\begin{proof}
Applying Lemma \ref{prod act} and using (\ref{sumVij}), the sum (\ref{sumVtilde}) is an eigenspace for $\mathbb{A}_{D}^{-1}\mathbb{A}_{D}^{*-1}\mathbb{A}_{D}\mathbb{A}_{D}^{*}$ with eigenvalue  $(1-r)^s$.
Its dimension  is obtained using Lemma \ref{lem12}.
\end{proof}

\noindent {\it Proof of Theorem \ref{lem8}}.
Combine (\ref{sumVij}), Lemma \ref{lem12}(i) and Theorem \ref{lem13}. \qquad \qquad \qquad$\square$\\


Recall the subconstituent algebra $\mathbb{T}$ for $H(D,r)$.
We now discuss the  irreducible $\mathbb{T}$-modules.
By \cite[p. 195]{3paper}, every  irreducible $\mathbb{T}$-module is thin.

\begin{lem} {\rm  \cite[p. 202]{3paper} } \label{lemr=t}
For the graph $H(D,r)$, let ${W}$ denote an irreducible $\mathbb{T}$-module. 
Then the endpoint of $W$ is equal to the dual endpoint of $W$.
\end{lem}

In the next lemma we refer to Definition \ref{defVeta}.

\begin{lem} \label{lem34}
For $0 \leq  \eta  \leq  D$,
the  subspace $\mathbb{V}_\eta$ is spanned by the irreducible $\mathbb{T}$-modules with displacement $\eta$.
\end{lem}

\begin{proof}
Let  $\mathbb{V}'_{\eta}$ denote the span of the irreducible $\mathbb{T}$-modules with displacement $\eta$.
We show $\mathbb{V}_{\eta} = \mathbb{V}'_{\eta}$.
By \cite[Theorem 6.2(i)]{split}, 
\begin{align} 
 \mathbb{V}'_{\eta} = \displaystyle  \sum_{\substack{0 \leq i,j \leq D\\ i+j = \eta + D}} \widetilde{\mathbb{V}}_{ij}.
\end{align}
By Lemma \ref{lemVeta}, $\mathbb{V}_\eta = \mathbb{V}'_{\eta}$.
\end{proof}

Note that $\mathbb{A}_{D}^{-1}\mathbb{A}_{D}^{*-1}\mathbb{A}_{D}\mathbb{A}_{D}^{*} \in \mathbb{T}$. 
Next, we describe the  action of $\mathbb{A}_{D}^{-1}\mathbb{A}_{D}^{*-1}\mathbb{A}_{D}\mathbb{A}_{D}^{*}$ on any irreducible $\mathbb{T}$-module.
We introduce some notation.

\begin{defn} \label{defFs}
{\rm For $-D \leq s \leq D$, let $\mathbb{F}_s \in Mat_{\mathbb{X}}(\mathbb{C}) $ denote the projection onto the eigenspace for $\mathbb{A}_{D}^{-1}\mathbb{A}_{D}^{*-1}\mathbb{A}_{D}\mathbb{A}_{D}^{*}$ associated with the eigenvalue  $(1-r)^s$.}
\end{defn}

We refer to Definition \ref{defFs}.
By linear algebra, $\mathbb{F}_s$ is a polynomial in $\mathbb{A}_{D}^{-1}\mathbb{A}_{D}^{*-1}\mathbb{A}_{D}\mathbb{A}_{D}^{*}$ and is therefore contained in $\mathbb{T}$. \\
Moreover,
$$\begin{array}{rcl}
\mathbb{F}_s \mathbb{F}_t &=&  \delta_{st} \mathbb{F}_s\qquad \quad (-D \leq s,t  \leq D),\\
\displaystyle  \sum_{s=-D}^{D} \mathbb{F}_s &=&  I, \\ [15pt]
\mathbb{A}_{D}^{-1}\mathbb{A}_{D}^{*-1}\mathbb{A}_{D}\mathbb{A}_{D}^{*} &=& \displaystyle  \sum_{s=-D}^{D} (1-r)^s \mathbb{F}_s.
\end{array}$$
Write $\mathbb{V}$ as an orthogonal direct sum of irreducible $\mathbb{T}$-modules :  $\mathbb{V} =   \sum_{{W}}{W}$.\\
For $-D \leq s \leq D$,
\begin{displaymath}
\mathbb{F}_s\mathbb{V} =   \sum_{{W}}\mathbb{F}_{s}{W}.
\end{displaymath}
In the above sum, each nonzero summand $\mathbb{F}_{s}{W}$ is the eigenspace for the action of $\mathbb{A}_{D}^{-1}\mathbb{A}_{D}^{*-1}\mathbb{A}_{D}\mathbb{A}_{D}^{*}$ on ${W}$ associated with the eigenvalue $(1-r)^s$.\\
 Let ${W}$ denote an irreducible $\mathbb{T}$-module. 
Then $W$ is the direct sum of the nonzero terms among $\{\mathbb{F}_{s}{W}\}_{s=-D}^{D}$.

\begin {thm} \label{thm36}
For the graph $H(D,r)$, let ${W}$ denote an irreducible $\mathbb{T}$-module with diameter $d$. 
Then for $-D \leq s \leq D$ the following are equivalent:
\begin{enumerate}
\item[\rm(i)] $\mathbb{F}_{s}{W} \neq 0$;
\item[\rm(ii)] $-d \leq s \leq d$ and $d-s$ is even.
\end{enumerate}
Now assume $\rm{(i)}$, $\rm{(ii)}$ hold. Then $\dim (\mathbb{F}_{s}{W}) = 1$.
\end{thm}

\begin{proof}
$\rm{(i) \Rightarrow (ii)}$ Let $\eta$ be the displacement of $W$.
By Lemma \ref{lem34}, $W \subseteq \mathbb{V}_\eta$.
We have $0 \neq \mathbb{F}_{s}{W} \subseteq W \subseteq \mathbb{V}_\eta$.
By the definition of displacement, 
\begin{align} \label{defeta}
\eta = \rho + \tau + d - D,
\end{align} 
where $\rho, \tau$ denote the endpoint and dual endpoint of $W$, respectively.
By Lemma \ref{lemr=t}, 
\begin{align} \label{r=t}
\rho = \tau.
\end{align} 
By Lemma \ref{lemVeta} and Theorem \ref{lem13}, there exist $0 \leq i,j \leq D$ such that 
 \begin{align} 
i+j &= \eta + D, \label{eq1}\\
j-i &=  s \label{eq2},
\end{align}
and $W$ is not orthogonal to $\widetilde{\mathbb{V}}_{ij}$ .
From (\ref{defeta}) minus (\ref{r=t}) plus (\ref{eq1}) minus (\ref{eq2}), we get
 \begin{align} \label{eqd-s}
d-s =2(i- \rho).
\end{align}
Therefore $d-s$ is even.
From (\ref{defeta}) plus (\ref{r=t}) plus (\ref{eq1}) plus (\ref{eq2}), we get
 \begin{align}  \label{eqd+s}
d+s =2(j- \tau).
\end{align}
By Definitions \ref{defVij}, \ref{defVijtilde} and the statement below (\ref{eq2}), we have $i \geq \rho$ and $j \geq \tau$.
By this and (\ref{eqd-s}), (\ref{eqd+s})  we obtain  $-d \leq s \leq d$. \\
$\rm{(ii) \Rightarrow (i)}$ Define the set $S=\{d, d-2, d-4, \ldots, -d\}$. Observe that $|S|= d+1$.
The $\mathbb{T}$-module $W$  is thin with diameter $d$, so it has dimension $d+1$.
Note that $W$ is the direct sum of the nonzero subspaces among $\{\mathbb{F}_{s}{W}\}_{s=-D}^{D}$.
These nonzero subspaces have dimension $1$ by \cite[Lemma 3.5]{split} and \cite[Lemma 3.8]{lin}.
By the earlier part of this proof, $\mathbb{F}_{s}{W} = 0$ unless $s \in S$.
It follows that  $\mathbb{F}_{s}{W} \neq 0$ for all $s \in S$.


\noindent Next assume that  \rm(i), \rm(ii) hold. 
We mentioned in the proof of $\rm{(ii) \Rightarrow (i)}$ that $\dim (\mathbb{F}_{s}{W}) = 1$.
\end{proof}

\begin {cor}
For the graph $H(D,r)$, let ${W}$ denote an irreducible $\mathbb{T}$-module with diameter $d$. 
Then the action of $\mathbb{A}_{D}^{-1}\mathbb{A}_{D}^{*-1}\mathbb{A}_{D}\mathbb{A}_{D}^{*}$ on ${W}$ is diagonalizable with eigenvalues
\begin{center}
 $(1-r)^s$ \qquad $-d \leq s \leq d$, \quad $d-s$ is even.
\end{center}
The corresponding eigenspaces all have dimension $1$.
\end{cor}

\begin{proof}
By Theorem \ref{thm36} and the comment above it.
\end{proof}

\section{Acknowledgement}
The author would like to thank Professor Paul Terwilliger for many valuable ideas and suggestions. This paper was written while the author was an Honorary Fellow at the University of Wisconsin-Madison (January 2017- January 2018) supported by the Development and Promotion of Science and Technology Talents (DPST) Project, Thailand.

\end{document}